\documentclass[12pt]{amsart}
\usepackage{amsmath,amssymb,amsbsy,amsfonts,amsthm,latexsym,
                        amsopn,amstext,amsxtra,euscript,amscd,mathrsfs,color,bm}
                        
\usepackage{float} 
\usepackage[english]{babel}
\usepackage{url}
\usepackage[colorlinks,linkcolor=blue,anchorcolor=blue,citecolor=blue,backref=page]{hyperref}

\usepackage{color}
\usepackage[norefs,nocites]{refcheck}

\renewcommand*{\backref}[1]{}
\renewcommand*{\backrefalt}[4]{%
    \ifcase #1 (Not cited.)%
    \or        (p.\,#2)%
    \else      (pp.\,#2)%
    \fi}

\begin{document}

\newtheorem{theorem}{Theorem}
\newtheorem{lemma}[theorem]{Lemma}
\newtheorem{example}[theorem]{Example}
\newtheorem{algol}{Algorithm}
\newtheorem{corollary}[theorem]{Corollary}
\newtheorem{prop}[theorem]{Proposition}
\newtheorem{definition}[theorem]{Definition}
\newtheorem{question}[theorem]{Question}
\newtheorem{problem}[theorem]{Problem}
\newtheorem{remark}[theorem]{Remark}
\newtheorem{conjecture}[theorem]{Conjecture}

\newcommand{\commM}[1]{\marginpar{%
\begin{color}{red}
\vskip-\baselineskip 
\raggedright\footnotesize
\itshape\hrule \smallskip M: #1\par\smallskip\hrule\end{color}}}

\newcommand{\commA}[1]{\marginpar{%
\begin{color}{blue}
\vskip-\baselineskip 
\raggedright\footnotesize
\itshape\hrule \smallskip A: #1\par\smallskip\hrule\end{color}}}
\def\xxx{\vskip5pt\hrule\vskip5pt}


\def\cA{{\mathcal A}}
\def\cB{{\mathcal B}}
\def\cC{{\mathcal C}}
\def\cD{{\mathcal D}}
\def\cE{{\mathcal E}}
\def\cF{{\mathcal F}}
\def\cG{{\mathcal G}}
\def\cH{{\mathcal H}}
\def\cI{{\mathcal I}}
\def\cJ{{\mathcal J}}
\def\cK{{\mathcal K}}
\def\cL{{\mathcal L}}
\def\cM{{\mathcal M}}
\def\cN{{\mathcal N}}
\def\cO{{\mathcal O}}
\def\cP{{\mathcal P}}
\def\cQ{{\mathcal Q}}
\def\cR{{\mathcal R}}
\def\cS{{\mathcal S}}
\def\cT{{\mathcal T}}
\def\cU{{\mathcal U}}
\def\cV{{\mathcal V}}
\def\cW{{\mathcal W}}
\def\cX{{\mathcal X}}
\def\cY{{\mathcal Y}}
\def\cZ{{\mathcal Z}}

\def\C{\mathbb{C}}
\def\F{\mathbb{F}}
\def\K{\mathbb{K}}
\def\G{\mathbb{G}}
\def\Z{\mathbb{Z}}
\def\R{\mathbb{R}}
\def\Q{\mathbb{Q}}
\def\N{\mathbb{N}}
\def\M{\textsf{M}}
\def\U{\mathbb{U}}
\def\P{\mathbb{P}}
\def\A{\mathbb{A}}
\def\p{\mathfrak{p}}
\def\n{\mathfrak{n}}
\def\X{\mathcal{X}}
\def\x{\textrm{\bf x}}
\def\w{\textrm{\bf w}}
\def\ovQ{\overline{\Q}}
\def\rank#1{\mathrm{rank}#1}
\def\wf{\widetilde{f}}
\def\wg{\widetilde{g}}
\def\comp{\hskip -2.5pt \circ  \hskip -2.5pt}
\def\({\left(}
\def\){\right)}
\def\[{\left[}
\def\]{\right]}
\def\<{\langle}
\def\>{\rangle}

\def\gen#1{{\left\langle#1\right\rangle}}
\def\genp#1{{\left\langle#1\right\rangle}_p}
\def\genPs{{\left\langle P_1, \ldots, P_s\right\rangle}}
\def\genPsp{{\left\langle P_1, \ldots, P_s\right\rangle}_p}

\def\e{e}

\def\eq{\e_q}
\def\fh{{\mathfrak h}}

\def\lcm{{\mathrm{lcm}}\,}

\def\l({\left(}
\def\r){\right)}
\def\fl#1{\left\lfloor#1\right\rfloor}
\def\rf#1{\left\lceil#1\right\rceil}
\def\mand{\qquad\mbox{and}\qquad}

\def\jt{\tilde\jmath}
\def\ellmax{\ell_{\rm max}}
\def\llog{\log\log}

\def\m{{\rm m}}
\def\ch{\hat{h}}
\def\GL{{\rm GL}}
\def\Orb{\mathrm{Orb}}
\def\Per{\mathrm{Per}}
\def\Preper{\mathrm{Preper}}
\def \S{\mathcal{S}}
\def\vec#1{\mathbf{#1}}
\def\ov#1{{\overline{#1}}}
\def\Gal{{\rm Gal}}

\newcommand{\bfalpha}{{\boldsymbol{\alpha}}}
\newcommand{\bfomega}{{\boldsymbol{\omega}}}

\newcommand{\Ch}{{\operatorname{Ch}}}
\newcommand{\Elim}{{\operatorname{Elim}}}
\newcommand{\proj}{{\operatorname{proj}}}
\newcommand{\h}{{\operatorname{h}}}

\newcommand{\hh}{\mathrm{h}}
\newcommand{\aff}{\mathrm{aff}}
\newcommand{\Spec}{{\operatorname{Spec}}}
\newcommand{\Res}{{\operatorname{Res}}}

\numberwithin{equation}{section}
\numberwithin{theorem}{section}

\def\house#1{{%
    \setbox0=\hbox{$#1$}
    \vrule height \dimexpr\ht0+1.4pt width .5pt depth \dimexpr\dp0+.8pt\relax
    \vrule height \dimexpr\ht0+1.4pt width \dimexpr\wd0+2pt depth \dimexpr-\ht0-1pt\relax
    \llap{$#1$\kern1pt}
    \vrule height \dimexpr\ht0+1.4pt width .5pt depth \dimexpr\dp0+.8pt\relax}}

\title[]
{On algebraic integers of bounded house and preperiodicity in polynomial semigroup dynamics}

\author[A. Ostafe] {Alina Ostafe}
\address{School of Mathematics and Statistics, University of New South Wales, Sydney NSW 2052, Australia}
\email{alina.ostafe@unsw.edu.au}

\author[M. Young] {Marley Young}
\address{School of Mathematics and Statistics, University of New South Wales, Sydney NSW 2052, Australia}
\email{marley.young@unsw.edu.au}

%


%

\begin{abstract} 
We consider semigroup dynamical systems defined by several polynomials over a number field $\K$, and the orbit (tree) they generate at a given point. We obtain finiteness results for the set of preperiodic points of such systems that fall in the cyclotomic closure of $\K$. More generally, we consider the finiteness of initial points in the cyclotomic closure for which the orbit contains an algebraic integer of bounded house. This work extends previous results for classical obits generated by one polynomial over $\K$ obtained initially by Dvornicich and Zannier (for preperiodic points), and then by Chen and Ostafe (for roots of unity and elements of bounded house in orbits).
\end{abstract}

\maketitle

\section{Introduction and statements}

\subsection{Motivation}

Let $\K$ be a field and let $\cF=\{f_1,\ldots,f_s\}\subset\K(X)$. We define recursively
\begin{equation}
\label{eq:tree}
\cF_1=\cF,\quad \cF_k=\bigcup_{i=1}^s f_i(\cF_{k-1}), \ k\ge 2,
\end{equation}
where $f_i(\cF_{k-1})=\{f_i \circ g \mid g\in\cF_{k-1}\}$, $i=1,\ldots,s$.

In other words, $\cup_{k\ge 1}\cF_k$ is the semigroup generated by $f_1,\ldots,f_s$ under composition. Accordingly, we shall speak of a {\it semigroup dynamical system}.

For an element $\alpha\in\ov\Q$, we define 
\begin{equation}
\label{eq:tree alpha}
\cF_k(\alpha)=\{f(\alpha)\mid f\in\cF_{k}\}, \quad k\ge 1,
\end{equation}
and the {\it orbit} of $\cF$ at $\alpha$ by
\begin{equation}
\label{eq:orb}
\Orb_\cF(\alpha)=\bigcup_{k\ge 0} \cF_k(\alpha).
\end{equation}

We note that for $s=1$, $\cF_k(\alpha)$ is just $k$-th iterate of $f_1$ at $\alpha$, that is, $f_1^{(k)}(\alpha)$, where $f_1^{(k)}$ represents the $k$-th compositional iterate of $f_1$.
Moreover, in this case, $\Orb_\cF(\alpha)=\{f_1^{(k)}(\alpha) \mid k\ge 0\}$ is just the classical forward orbit of $f_1$ at $\alpha$.

For $s=2$, $\Orb_\cF(\alpha)$ can be seen as the binary tree represented below, in which the level $k\ge 1$ is given by the set $\cF_k(\alpha)$.
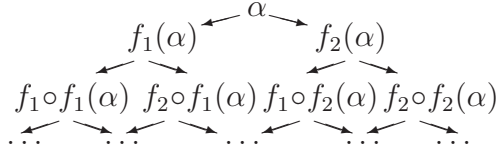
\begin{figure}[h]
\setlength{\unitlength}{1cm}
\vskip 3pt 
\begin{picture}(170,1.6)(9,1.2)
\put(15.1,2.9){$\alpha$}
\put(15.0,2.9){\vector(-4,-1){0.5}}
\put(15.4,2.9){\vector(4,-1){0.5}}
\put(13.5,2.5){$f_1(\alpha)$}  
\put(16.0,2.5){$f_2(\alpha)$}  
\put(13.6,2.3){\vector(-3,-1){0.5}}
\put(13.8,2.3){\vector(3,-1){0.5}}
\put(16.4,2.3){\vector(-3,-1){0.5}}
\put(16.6,2.3){\vector(3,-1){0.5}}
\put(12.0,1.7){$f_1\comp f_1(\alpha)$}  
\put(13.7,1.7){$f_2 \comp f_1(\alpha)$}  
\put(15.3,1.7){$f_1 \comp f_2(\alpha)$}  
\put(16.9,1.7){$f_2\comp f_2(\alpha)$}  
\put(12.8,1.5){\vector(3,-1){0.5}}
\put(12.6,1.5){\vector(-3,-1){0.5}}
\put(14.2,1.5){\vector(3,-1){0.5}}
\put(14,1.5){\vector(-3,-1){0.5}}
\put(15.8,1.5){\vector(-3,-1){0.5}}
\put(16.0,1.5){\vector(3,-1){0.5}}
\put(17.2,1.5){\vector(-3,-1){0.5}}
\put(17.4,1.5){\vector(3,-1){0.5}}
\put(11.9,1.2){\ldots}
\put(13.2,1.2){\ldots}
\put(14.8,1.2){\ldots}
\put(16.4,1.2){\ldots}
\put(17.6,1.2){\ldots}
\end{picture}
\end{figure}

The study of semigroup dynamical systems goes back to work of Silverman~\cite{Si93} who investigated the finiteness of $S$-integers in orbits generated by a finite set of rational functions, and more recently appears in several works addressing important questions in arithmetic dynamics  such as the Mordell-Lang Conjecture for endomorphisms of semiabelian varieties~\cite{GTZ1}, the orbit intersection problem~\cite{GN,GTZ}, see also~\cite{YZ}, the  irreducibility of compositions of polynomials~\cite{FGS,GMS,HBM}.

In this paper we restrict ourselves to the case when $f_1,\ldots,f_s \in \K[X]$.

We  reserve $|\alpha|$ for  the usual  absolute value of  $\alpha \in \C$ and use
$\house{\strut\alpha}$ for the {\it house\/} of $\alpha$, which is the maximum of absolute values $|\sigma(\alpha)|$ 
of the conjugates $\sigma(\alpha)$ over $\Q$ of   $\alpha\in \ovQ$.

In this paper we are looking at the presence of algebraic integers of bounded house in the sets~\eqref{eq:tree alpha} for univariate polynomials defined over a number field $\K$, generalising recent work of~\cite{Chen,DZ,Ost}. In particular, we prove that for any $A \geq 1$, unless the polynomials $f_1,\ldots,f_s$ are special (in a sense to be defined in the next section), there are at most finitely many elements $\alpha$ in the cyclotomic closure $\K^c$ such that the sets $\cF_k(\alpha)$ contain at least one algebraic integer of bounded by $A$ house for any $k\ge 1$.

A special case of elements of bounded house is given by {\it preperiodic points} of a given polynomial. In this case, several points of view can be adopted in generalising the notion of preperiodicity from classical dynamics ($s=1$): one in which two elements are equal along a connected path in the tree and one in which two elements are equal in the whole tree at different levels. For precise definitions see Section~\ref{sec:per} below.

The initial motivation of this work is a result of Dvornicich and Zannier~\cite[Theorem~2]{DZ}  on  the finiteness of preperiodic points of a univariate polynomial defined over a  number field $\K$, falling in the cyclotomic closure $\K^c$; see also~\cite[Theorem~34]{KaSi} for a multivariate case. More concretely,  they prove that a polynomial $f\in\K[X]$ has only finitely many preperiodic points in $\K^c$, unless $f$ is conjugated to monomials or Chebyshev polynomials. We shall obtain a similar result in the case of semigroup dynamics.

Our approach extends that of~\cite{Ost}, and its generalisation due to Chen~\cite{Chen}. Namely, it utilizes~\cite[Theorem~2.5]{Chen}, which follows the ideas and technique of~\cite{DZ}, and is based on the toric analogue of the Manin--Mumford conjecture  (also called the {\it torsion points theorem}) proved by Laurent~\cite{Laurent}  
(a more elementary proof was also given by Sarnak and Adams~\cite{SA}),  and an extension of Loxton's result~\cite[Theorem~1]{Loxt} in representing cyclotomic integers as short combination of roots of unity~\cite[Theorem~L]{DZ}. This is used in conjunction with various bounds on valuations of elements in (semigroup) orbits.

Our  arguments stem  from the previous investigations~\cite{Chen,DZ,Ost} that apply to classical orbits defined by just one univariate polynomial, however they require some modifications, which are 
not always straight-forward and present new technical challenges. We believe that the methods and auxiliary results we develop here   can be  of use for some other questions on polynomial 
semigroups. Thus, we also outline some open problems and future research directions.

We conclude by noting that studying points of bounded house in orbits of
algebraic dynamical systems can be seen as a Bogomolov-type generalisation
of studying the set of preperiodic points. In this context, we note the Dynamical Bogomolov Conjecture~\cite[Conjecture 3.62]{Silv}, generalising the Dynamical Manin-Mumford Conjecture~\cite[Conjecture 3.58]{Silv}, about points of small canonical height on non-preperiodic algebraic varieties. This point of view also appears in `unlikely intersections' problems in arithmetic dynamics, where for example, instead of looking at simultaneous preperiodic points for two dyanmical systems, one looks at common points of small canonical height, see~\cite[Conjectures 11.3 and 15.4]{BdMIJMST} and references therein.

\subsection{Notation, conventions and definitions} 
\label{sec:def}

We use the following notations:
\begin{itemize}
\item[$\quad \diamond$] $\ovQ$: the  algebraic  closure of $\Q$;
\item[$\quad \diamond$] $\U$: the set of all roots of unity in $\C$; 
\item[$\quad \diamond$] $\K$: a number field; 
\item[$\quad \diamond$] $\K^c=\K(\U)$: the cyclotomic closure of $\K$; 
\item[$\quad \diamond$] $\cF=\{f_1,\ldots,f_s\}$: a set of distinct polynomials $f_i\in\K[X]$ of degrees $d_i\ge 2$, $i=1,\ldots,s$;
\item[$\quad \diamond$] $\cH_A$: the set of all algebraic integers of house at most $A$;
\item[$\quad \diamond$] $T_d$: the Chebyshev polynomial of degree $d$; it is uniquely determined by the equation $T_d(x+x^{-1})=x^d+x^{-d}$.
\end{itemize}

We recall that by the celebrated result of Kronecker any algebraic integer of house $1$ is a root of unity,
that is, $\cH_1 = \U$.

For $A \geq 1$ and a set of polynomials $\cF=\{f_1,\ldots,f_s\} \subset \K[X]$, we define
\begin{equation}
\label{eq:SA}
\cS_A(\cF) = \left \{ \alpha \in \K^c \mid \bigcup_{k \geq 1} \cF_k(\alpha) \cap \cH_A \neq \emptyset \right \}.
\end{equation}

Moreover, for $i_1\ldots,i_k\in\{1,\ldots,s\}$, we use the notation
\begin{equation}
\label{eq:fk}
f_{i_1\ldots i_k}=f_{i_k}\circ f_{i_{k-1}} \circ \cdots \circ f_{i_1},
\end{equation}
and define $f_{i_1 \ldots i_k}(X)=X$ if $k=0$.

\begin{definition} [\sl{Special set of polynomials}]
\label{def:Srf Sys}
Let $\cF=\{f_1,\ldots,f_s\} \subset \K[X]$ be a set of polynomials with $\deg f_i=d_i$, $i=1,\ldots,s$. We call the set  $\cF$   {\sl special} if  there exist linear polynomials $\ell, \ell_1,\ell_2\in\ov\Q[X]$ such that at least one of the following conditions  is satisfied:
\begin{itemize}
\item For some $i=1,\ldots,s$, one has
$$f_i=\ell \circ X^{d_i} \circ \ell^{-1} \quad \textrm{or} \quad f_i=\ell \circ (\pm T_{d_i}) \circ \ell^{-1};$$ 
\item For some $1\le i\ne j\le s$, one has
$$
f_i=\ell_1\circ X^{d_i}\circ \ell \quad \textrm{and} \quad f_j=\ell^{-1}\circ X^{d_j}\circ \ell_2
$$
or
$$
f_i=\ell_1\circ T_{d_i}\circ \ell \quad \textrm{and} \quad f_j=\ell^{-1}\circ T_{d_j}\circ \ell_2.
$$
\end{itemize}

If none of these conditions is satisfied for any linear polynomials $\ell, \ell_1,\ell_2\in\ov\Q[X]$, then we call the set $\cF$  {\sl non-special}. 
\end{definition}

\begin{remark} 
We note that the second condition of Definition~\ref{def:Srf Sys} is equivalent, by~\cite[Corollary 2.9]{MZ}, to saying that there exist linear polynomials $\ell_1,\ell_2\in\ov\Q[X]$ such that $\ell_1\circ (f_i\circ f_j)\circ\ell_2$ is either $X^{d_id_j}$ or $T_{d_id_j}(X)$.
\end{remark}
\begin{remark} 
We also note that for $d_i>2$, $i=1,\ldots,s$, we can simplify Definition~\ref{def:Srf Sys} by requesting 
only the second condition also for $i=j$. Indeed, this follows from Lemma~\ref{lem:mon Cheb} below (which in turn follows from~\cite[Lemma 3.13]{MZ} or~\cite[Lemma 3.9]{GTZ}).
\end{remark}

\subsection{Periodicity and preperiodicity in semigroup dynamics}
\label{sec:per}

Let $\cF=\{f_1,\ldots,f_s\}\subset\K[X]$. In this section we define  the notion of {\it preperiodic} point in the orbit (tree) $\Orb_\cF(\alpha)$ defined by~\eqref{eq:orb} of an element $\alpha\in\ov\Q$, extending the usual notion from classical dynamics generated by only one polynomial. In the semigroup case, there can be multiple ways to interpret this notion. A first notion of periodicity was considered in~\cite{Kaw} where a point $\alpha$ is said to be periodic if $\Orb_\cF(\alpha)$ is finite. We consider below two more general definitions. 

For this we recall the notation~\eqref{eq:tree alpha}.

Firstly, one can view preperiodic points as those $\alpha\in\ov\Q$ such that $$f_{i_1 \ldots i_k}(\alpha) \in \cF_\ell(f_{i_1 \ldots i_k}(\alpha))$$ for some $k \geq 0$, $\ell \ge 1$ and $i_1,\ldots,i_k \in \{1,\ldots,s\}$. That is, two elements in a given connected path in the tree of $\alpha$ are equal, or in other words, some element in the orbit of $\alpha$ is ``periodic".

We denote by $\Pi(\cF)$  the set of $\alpha\in\ov\Q$ being preperiodic in this sense, that is,
\begin{equation}
\label{eq:p1}
\Pi(\cF)=\{\alpha\in\ov\Q \mid f_{i_1 \ldots i_k}(\alpha) \in \cF_\ell(f_{i_1 \ldots i_k}(\alpha)) \textrm{ for some $k\ge 0, \ell\ge 1$}\}.
\end{equation}

Secondly, and more generally, one can call $\alpha$ a preperiodic point if we have an arbitrary collision (at different levels in the tree) $$\cF_k(\alpha) \cap \cF_n(\alpha) \neq \emptyset$$ for some $k \neq n$.

We denote by $\ov\Pi(\cF)$ the set of $\alpha\in\ov\Q$ being preperiodic in this second sense, that is,
\begin{equation}
\label{eq:p2}
\ov\Pi(\cF)=\{\alpha\in\ov\Q \mid \cF_k(\alpha) \cap \cF_n(\alpha) \neq \emptyset \textrm{ for some $k\ne n$}\}.
\end{equation}
We have the inclusion $\Pi(\cF)\subseteq \ov\Pi(\cF)$.

\subsection{Main results}
\label{sec:main}

Our first aim is to describe sets of polynomials $\cF=\{f_1,\ldots,f_s\}$ defined over $\K$ for which the set $\cS_A(\cF)$ defined by~\eqref{eq:SA} is finite. The key condition for our results is that the images of each $f_i$ have finite intersection with $\cH_A$ (or in other words that $\cF(\K^c)\cap \cH_A$ is finite). We later discuss (see  Lemma~\ref{lem:chen}~and~Remark~\ref{rem:PAAvoid}) when this is satisfied. In particular, we see from Remark~\ref{rem:PAAvoid} that  this finiteness condition is generically 
satisfied by all polynomials of sufficiently large degree.

\begin{theorem}
\label{thm:KcOrb}
Let $\cF=\{f_1,\ldots,f_s\} \subset \K[X]$ be a non-special set of polynomials of respective degrees $d_i \ge 2$, and let $A \ge 1$ be a real number. If $f_i(\K^c) \cap \cH_A$ is finite for $i=1,\ldots,s$, then so is $\cS_A(\cF)$ defined by~\eqref{eq:SA}.
\end{theorem}

\begin{remark}
If $s=1$ in Theorem~\ref{thm:KcOrb}, we obtain the result of~\cite{Chen}.
\end{remark}

Using the cyclotomic version of the Hilbert's Irreducibility Theorem due to Dvornicich and Zannier~\cite[Corollary 1]{DZ}, one can deduce immediately from Theorem~\ref{thm:KcOrb} the following conclusion about the presence of roots of unity in orbits of semigroup dynamics, which extends~\cite[Theorem 1.2]{Ost} when $s=1$.

\begin{corollary}
\label{cor:r o u}
Let $\cF=\{f_1,\ldots,f_s\} \subset \K[X]$ be a non-special set of polynomials of respective degrees $d_i \ge 2$, such that $f_i(X)-Y^{m_i}$ as polynomials in $X$ do not have a root in $\K^c(Y)$ for all $m_i\le d_i$ and $i=1,\ldots,s$. Then $\cS_1(\cF)$ is finite.
\end{corollary}

We also have results analogous to~\cite[Theorem~2]{DZ}, which consider the finiteness of preperiodic points in the semigroup case falling in $\K^c$.

\begin{theorem}
\label{thm:KcPreper}
Let $\cF=\{f_1,\ldots,f_s\} \subset \K[X]$ be a non-special set of polynomials of respective degrees $d_i \ge 2$. Then $\Pi(\cF)\cap \K^c$ is finite, where $\Pi(\cF)$ is defined by~\eqref{eq:p1}.
\end{theorem}

\begin{remark}
Since $\Pi(\cF)$ is a set of uniformly bounded house only in terms of $f_1,\ldots,f_s$ and $\K$ (which follows from the proof of Theorem~\ref{thm:KcPreper}), we note that Theorem~\ref{thm:KcPreper} also follows directly from Theorem~\ref{thm:KcOrb} under the condition that $f_i(\K^c) \cap \cH_A$ is finite for $i=1,\ldots,s$. However, this latter condition is not needed for achieving the finiteness of $\Pi(\cF)\cap\K^c$,  see the proof of Theorem~\ref{thm:KcPreper}.
\end{remark}

We also prove that the set $\ov\Pi(\cF)$ of preperiodic points as defined by~\eqref{eq:p2} is a set of bounded house. For our methods to work with this definition, we restrict to the case where the polynomials $f_1,\ldots, f_s$ all have the same degree, in order to control the house of an arbitrary preperiodic point.

In fact, we generalise further to points which have a semigroup iterate which can be written as a linear combination of elements in earlier levels of the tree, with coefficients of bounded house. This gives a result in the vein of Northcott's theorem, see \cite[Theorem~3.12]{Silv}.

For $A \geq 1$ and $\alpha \in \ov \Q$, we define
$$
\cL_{n,A}(\cF;\alpha) = \left\{ \sum_{k=0}^n \sum_{i_1, \ldots, i_k = 1}^s \gamma_{i_1 \ldots i_k} f_{i_1 \ldots i_k}( \alpha ) \mid \gamma_{i_1 \ldots i_k} \in \cH_{A^{d^n}} \right\},
$$
where $f_1, \ldots, f_s \in \K[X]$ all have degree $d$, and set
$$
\Sigma_A(\cF)=\{\alpha\in\ov\Q \mid \cF_n(\alpha) \cap \cL_{n-1,A}(\cF;\alpha) \neq \emptyset \textrm{ for some $n\ge 1$}\}.
$$
We have the inclusion $\Pi(\cF) \subseteq \ov \Pi(\cF) \subseteq \Sigma_A(\cF)$ for any $A \geq 1$.

\begin{theorem}
\label{thm:KcPreper2}
Let $\cF=\{f_1,\ldots,f_s\} \subset \K[X]$ be a set of polynomials of the same degree $d \ge 3$ and let $A \geq 1$. Then  $\Sigma_A(\cF)$ is a set of bounded house, and there exists a positive integer $D$, depending only on $f_1,\ldots,f_s$ and $\K$, such that $D\Sigma_A(\cF)$ is a set of algebraic integers.
\end{theorem}

\begin{remark} (Polynomials of different degrees)
We note that, if $\deg f_i\ne \deg f_j$ for some $i\ne j$ then more conditions are needed for Theorem~\ref{thm:KcPreper2} to hold. For example, one needs to ensure that the polynomials $f_1,\ldots,f_s$ are compositionally independent, that is,
$$
f_{i_1 \ldots i_n}\ne f_{j_1 \ldots j_k}
$$
for any $k\ne n$ and any $i_1 \ldots i_n,j_1 \ldots j_k\in\{1,\ldots,s\}$. We note that when $\deg f_i=d$ for all $i=1,\ldots,s$, this is automatically satisfied since the polynomials in the equation above have different degrees.
\end{remark}

\begin{remark}
Let $f_1,f_2\in\K[X]$ be distinct polynomials. It is still of interest studying the subset of $\ov\Pi(\cF)$  defined by
$$
\{\alpha\in\ov\Q \mid f_1^{(k)}(\alpha)=f_2^{(t)}(\alpha) \textrm{ for some $k,t\in\N$, $k\ne t$}\}.
$$
In particular, Theorem~\ref{thm:KcPreper2} implies that if $\deg f_1=\deg f_2\ge 3$, then this is a set of bounded house. 

We also note that, under the condition $\deg f_1=\deg f_2\ge 2$, the above set is a set of bounded height. Indeed, it follows directly from~\cite[Theorem~3.11]{Silv} that there exist constants $C_i=C(f_i)$, $i=1,2$, such that for all $k\geq 1$ and all $\alpha\in\ov\Q$ we have
\begin{equation}\label{eq:height equiv}
d^k (h(\alpha)- C_i) \leq h(f_i^{(k)}(\alpha)) \leq d^k (h(\alpha) + C_i).
\end{equation}
Let $\alpha\in\ov\Q$ be such that $f_1^{(k)}(\alpha)=f_2^{(t)}(\alpha)$, and we may assume $k>t$. The above inequality gives
$$
d^k(h(\alpha)-C_1) \le d^t(h(\alpha)+C_2),
$$
from where one obtains that
$$
h(\alpha)\le \frac{d^t}{d^k-d^t}C_2+\frac{d^k}{d^k-d^t}C_1\le C_2+2C_1,
$$
where the last inequality holds since $d\ge 2$. 

A similar computation as above was done (using canonical heights) in the proof of~\cite[Theorem 1.7]{deMar}. Moreover,~\cite[Proposition 5.2]{deMar} shows that the above set, including also collisions at same level $k=t$ (which is in fact the most difficult case), is of bounded height for the polynomials $f_1(x)=3X^2+5$ and $f_2(X)=X^2$. 
\end{remark}

\begin{remark}{(Rational function analogues)}
One can likely extend Theorems~\ref{thm:KcOrb},~\ref{thm:KcPreper} and~\ref{thm:KcPreper2} to rational functions $h_i=f_i/g_i \in \K(X)$, $i=1,\ldots,s$, for which $\infty$ is an \emph{attracting point}. That is, $\deg f_i > \deg g_i + 1$ (see \cite{Chen, Ost}). However, as for the case $s=1$, these techniques do not seem to work when $\deg f_i \le \deg g_i + 1$.
\end{remark}

\section{Preliminaries}

\subsection{Polynomial decompositions}
In this section we describe the polynomials for which the second iterate is, up to composition on both sides with linear polynomials, a monomial or Chebyshev polynomial, which is given by~\cite[Lemma 3.9]{GTZ}. This will motivate the first condition of Definition~\ref{def:Srf Sys}. 

\begin{lemma}
\label{lem:mon Cheb}
Let $f\in\K[X]$ be of degree $d\ge 2$ and let $\ell_1,\ell_2\in\ov\Q[X]$ be linear polynomials.
\begin{itemize}
\item[(i)] If $f^{(2)}=\ell_1\circ X^{d^2}\circ\ell_2$, then there exists a linear polynomial $\ell\in\ov\Q[X]$ such that
$$
f=\ell\circ X^d\circ \ell^{-1}.
$$
\item[(ii)] If $f^{(n)}=\ell_1\circ T_{d^n}\circ\ell_2$ and $\max\{d,n\}>2$, then 
$$
f=\ell_1\circ (\pm T_d)\circ \ell_1^{-1}.
$$
\end{itemize}
\end{lemma}

\begin{proof}
Part (ii) is given by~\cite[Lemma 3.9]{GTZ}, so we prove only (i).
Since $f^{(2)}=\ell_1\circ X^{d^2}\circ\ell_2$, applying~\cite[Lemma 3.9]{GTZ}, we obtain that there exists an element $\alpha\in\ov\Q^*$ such that
\begin{equation}
\label{eq:alpha}
f=\ell_1\circ \alpha X^d\circ \ell_1^{-1}.
\end{equation}
Let now $\ell\in\ov\Q[X]$ such that
$$
\ell\circ f\circ \ell^{-1}(X)=X^d+a_{d-2}X^{d-2}+\cdots+a_0.
$$
Composing~\eqref{eq:alpha} with $\ell$ and $\ell^{-1}$, we obtain
$$
\ell\circ f\circ \ell^{-1}= \cL\circ \alpha X^d\circ \cL^{-1},
$$
where $\cL=\ell\circ\ell_1$. We write $\cL=uX+v$ for some $u,v\in\ov\Q$, $u\ne 0$. Thus, the equation above becomes
$$
X^d+a_{d-2}X^{d-2}+\cdots+a_0=\alpha u^{1-d}(X-v)^d+v.
$$
Comparing the coefficients of $X^{d-1}$ in the left and right hand sides above, we conclude that $v=0$, which implies that $a_i=0$, $i=0,\ldots,d-2$. We thus obtain that $\ell\circ f\circ \ell^{-1}=X^d$, which concludes this case.
\end{proof}

\subsection{Representations via linear combinations of roots of unity}
Loxton~\cite[Theorem~1]{Loxt} proved that any algebraic integer $\alpha$ contained in some cyclotomic field has a short representation as a sum of roots of unity, that is, $\alpha=\sum_{i=1}^b\zeta_i$, where $\zeta_1,\ldots,\zeta_b\in\U$, and $b\le \cR(\house{\alpha})$ for a suitable function  $\cR:\R_{+}\to \R_{+}$. We refer to such a function $\cR$ as a {\it Loxton function}.
One can take $\cR(x)\ll_{\epsilon}x^{2+\epsilon}$ (see~\cite[Theorem 1]{Loxt}).

Dvornicich and Zannier~\cite[Theorem~L]{DZ} extended the result of Loxton~\cite[Theorem~1]{Loxt} to algebraic integers contained in a cyclotomic extension of a given number field, which we present below. 

\begin{lemma}
\label{lem:loxton}
There exists a number $B$ and a finite set $E\subset \K$ with $\#E\le [\K:\Q]$  such that any algebraic integer $\alpha\in \K^c$ 
can be written as $\alpha=\sum_{i=1}^bc_i\xi_i$, where $c_i\in E$, $\xi_i\in\U$ and $b\le \# E\cdot \cR(B\house{\alpha})$, where $\cR:\R\to \R$ is any Loxton function.
\end{lemma}

For each number field $\K$ and $B$ and $E$ be as in Lemma~\ref{lem:loxton}, we define the function 
\begin{equation}
\label{eq:LK}
L_\K:\R_{>0}\to \R_{>0},\quad L_\K(t)=\#E\cdot \cR(B\cdot t).
\end{equation}

We then have the following result due to Chen~\cite[Theorem~2.5]{Chen}. 

\begin{lemma}
\label{lem:chen}
Let $f \in \K(X)$ be nonconstant and let $B$ and $E$ be as in Lemma~\ref{lem:loxton}. Then 
$$
\{ \alpha \in \K^c \mid f(\alpha) \in \cH_A \}
$$
is a finite set unless there exist a nonconstant $S \in \K^c(X)$, integers $n_i$, roots of unity $\beta_i \in \U$, and $e_i \in E$ such that
$$ f(S(X))=\sum_{i=1}^b \beta_i e_i X^{n_i} ,\quad b \leq L_\K(A). $$
\end{lemma}

Note that, if $f \in \K^c[X]$, the rational function $S$ in Lemma~\ref{lem:chen} must necessarily be a Laurent polynomial, as $f(S(X))$ can only have 0 as a pole. 

\begin{remark}
\label{rem:PAAvoid}
We note that ``most'' polynomials of large degree satisfy the condition of Theorem~\ref{thm:KcOrb}. Indeed, in the polynomial case, the following was proved in~\cite[Theorem~2.7]{Chen}: 
Let $A \geq 1$ and let $f \in \K[X]$ be nonconstant. Suppose $\deg f > (2 L_\K(A)+1)^2$. Then $f(\K^c) \cap \cH_A $ is finite unless, for some $a,b,c\in\K^c$,  $f(aX+b+cX^{-1})$ is a Laurent polynomial of length at most $L_\K(A)$  as in Lemma~\ref{lem:chen}.
\end{remark}

\subsection{The size of elements in orbits}
\label{sizegrow}
In this section we prove some useful simple facts about the size of compositions of polynomials that are similar to the ones in~\cite[Section~2.3]{Ost}.

\begin{lemma}
\label{lem:growcomp}
Let $f_i\in\K[X]$, $i=1,\ldots,s$, be polynomials of degrees $d_i \geq 2$. We write $f_i=a_{i,d_i}X^{d_i}+a_{i,d_i-1}X^{d_i-1}+\cdots+a_{i,0}\in \K[X]$
and let  $\alpha\in \ov\Q$ be such that 
\begin{equation}
\label{eq:alphacondpoly}
|\alpha|_v>\max_{\substack{i=1,\ldots,s\\j=0,\ldots,d_i-1}} \{1,|a_{i,j}|_v |a_{i,d_i}|_v^{-1}, |a_{i,d_i}|_v^{-1}\}
\end{equation} 
for some non-archimidean absolute value $|\cdot|_v$ of $\K$ (normalised in some way and extended to $\ov\K=\ov\Q$). Then 
$$
|f_{i_1\ldots i_n}(\alpha)|_v> |f_{i_1\ldots i_{n-1}}(\alpha)|_v,\quad i_1,\ldots,i_n\in\{1,\ldots,s\},\ n\ge 1,
$$
where $f_{i_1\ldots i_n}$ is defined by~\eqref{eq:fk}.
\end{lemma}

\begin{proof}
The proof follows the same computations as in~\cite[Lemma~2.5]{Ost} by induction on $n\ge 1$. However, as~\cite[Lemma~2.5]{Ost} appears only for monic polynomials, we repeat the computations. For $n=1$ we need to prove that $|f_{i_1}(\alpha)|_v>|\alpha|_v$. 
We note that 
\begin{align*}
 |a_{i_1,d_{i_1}}\alpha^{d_{i_1}}-f_{i_1}(\alpha)|_v & \le\max_{j = 0,\ldots, d_{i_1}-1} |a_{i_1,j} \alpha^{j}|_v  \\
 & = \max_{j = 0,\ldots, d_{i_1}-1} |a_{i_1,j}|_v |\alpha|_v^{j} < |a_{i_1,d_{i_1}} |_v |\alpha|_v^{d_{i_1}},
\end{align*}
where the last inequality follows from~\eqref{eq:alphacondpoly}. Hence, 
\begin{align}
|f_{i_1}(\alpha)|_v &=\max\{ |a_{i_1,d_{i_1}}\alpha^{d_{i_1}}-f_{i_1}(\alpha)|_v, |a_{i_1,d_{i_1}} |_v |\alpha|_v^{d_{i_1}}\} \notag \\
& =|a_{i_1,d_{i_1}}|_v|\alpha|_v^{d_{i_1}}>|\alpha|_v^{d_{i_1}-1}. \label{eq:grownonarch}
\end{align}
The result now follows since $d_{i_1} \ge 2$.

We assume now the statement true for iterates up to $n-1$.
Hence for 
$\beta = f_{i_1\ldots i_{n-1}}(\alpha)$ we have 
\begin{align*}
|\beta|_v  = |f_{i_1\ldots i_{n-1}}(\alpha)|_v > \ldots &> |f_{i_1}(\alpha)|_v > |\alpha|_v\\
& >\max_{\substack{i=1,\ldots,s\\j=0,\ldots,d_i-1}} \{1,|a_{i,j}|_v |a_{i,d_i}|_v^{-1}, |a_{i,d_i}|_v^{-1}\}.
\end{align*}
Applying the same argument as above with $\beta$ instead of $\alpha$ and with $f_{i_n}$ instead of $f_{i_1}$, we obtain the result.
\end{proof}

In the proof of Theorem~\ref{thm:KcOrb} we  
apply Lemma~\ref{lem:loxton} to values of polynomials from $\cF_k$ defined by~\eqref{eq:tree} at elements of $\K^c$. 
For this reason we need to control the growth of the house of such compositions, which is presented below.
However it is convenient to start with estimating the absolute value of such elements, which is an Archimedean version of Lemma~\ref{lem:growcomp}.
\begin{lemma}
\label{lem:growabsval}
Under the notation of Lemma~\ref{lem:growcomp}, let
$\alpha\in \C$ be such that 
$$
|\alpha|>1+\max_{i=1,\ldots,s}|a_{i,d_i}|^{-1}\l(1+\sum_{j=0}^{d_i-1}|a_{i,j}|\r).
$$
Then 
$$
|f_{i_1\ldots i_n}(\alpha)|> |f_{i_1\ldots i_{n-1}}(\alpha)|,\quad i_1,\ldots,i_n\in\{1,\ldots,s\},\ n\ge 1.
$$
\end{lemma}
\begin{proof} We again use induction over $n$. For $n=1$, by the triangle inequality, we have $|f_{i_1}(\alpha)| \geq |a_{i_1, d_{i_1}} \alpha^{d_{i_1}}| - |f_{i_1}(\alpha) - a_{i_1,d_{i_1}} \alpha^{d_{i_1}} |.$ Now,
\begin{align*}
|f_{i_1}(\alpha) - a_{i_1,d_{i_1}} \alpha^{d_{i_1}} | & = |a_{i_1, d_{i_1}-1} \alpha^{d_{i_1}-1} + \ldots + a_{i_1,0} | \\
 & \leq |\alpha|^{d_{i_1}-1} \sum_{j=0}^{d_{i_1}-1} |a_{i_1,j} |,
\end{align*}
where the last inequality follows since $|\alpha| > 1$. We conclude that
\begin{align} \label{eq:growabs}
|f_{i_1}(\alpha)| & \ge |\alpha|^{d_{i_1}-1} \l( |a_{i_1,d_{i_1}}| |\alpha| - \sum_{j=0}^{d_{i_1}-1} |a_{i_1,j} | \r) \\
& > |\alpha|^{d_{i_1}-1}, \notag
\end{align}
which gives the conclusion for $n=1$, since $d_{i_1} \geq 2$. The implication from $n-1$ to $n$ follows the exact same lines.
\end{proof}

\begin{corollary}
\label{cor:growhouse}
Under the notation of Lemma~\ref{lem:growcomp}, let $A\in\R$ be positive and define
$$
L=\max_{\sigma}\left\{ 1 +\max_{i=1,\ldots,s}|\sigma(a_{i,d_i}^{-1})|\l(1+\sum_{j=0}^{d_i-1}|\sigma(a_{i,j})|\r),A\right\},
$$
where the maximum runs over all embeddings $\sigma$ of $\K$ in $\C$. Let $\alpha\in \ov\Q$ be such that, for some $k\ge 1$ and some $i_1,\ldots,i_k\in\{1,\ldots,s\}$, we have 
 $\house{f_{i_1\ldots i_k}(\alpha)}\le A$.
Then $\house{f_{i_1\ldots i_\ell}(\alpha)}\le L$ for all $\ell<k$. 
\end{corollary}

\begin{proof}
This follows
immediately from Lemma~\ref{lem:growabsval}, see also~\cite[Corollary~2.8]{Ost}. Indeed, assume that $\house{f_{i_1\ldots i_\ell}(\alpha)}> L$ for some $\ell<k$. This means that there exists a conjugate of $f_{i_1\ldots i_\ell}(\alpha)$, which we denote by
$\sigma\l(f_{i_1\ldots i_\ell}(\alpha)\r)$, such that $|\sigma\l(f_{i_1\ldots i_\ell}(\alpha)\r)|> L$. We note that $\sigma\l(f_{i_1\ldots i_\ell}(\alpha)\r)=\sigma(f_{i_1\ldots i_\ell})(\sigma(\alpha))$, where $\sigma(f_{i_1\ldots i_\ell})=\sigma(f_{i_1}) \circ \ldots \circ\sigma(f_{i_{\ell}})$ is the composition of polynomials $f_{i_j}$, $j\le \ell$, in which we replace the coefficients $a_{i,l}$ of $f_i$ by $\sigma(a_{i,l})$. We apply now recursively Lemma~\ref{lem:growabsval} with the polynomials $\sigma(f_{i_j})$, $j=\ell+1,\ldots,k$, and the point $\sigma\l(f_{i_1\ldots i_\ell}(\alpha)\r)$ to conclude that $|\sigma\l(f_{i_1\ldots i_k}(\alpha)\r)|>A$, which is a contradiction with our assumption.
\end{proof}

We also require the following, which will allow us to apply Lemma~\ref{lem:chen} in the proof of Theorem~\ref{thm:KcOrb}.

\begin{lemma}
\label{lem:algint}
Let $f_1,\ldots,f_s \in \K[X]$. Then there exists a positive integer $D$, depending only on $f_1,\ldots,f_s$ such that for any $i_1,\ldots,i_n \in \{1,\ldots,s\}$ and $\alpha \in \K^c$, if $f_{i_1 \ldots i_n}(\alpha)$ is an algebraic integer, then $D \alpha$ and $D f_{i_1 \ldots i_r}(\alpha)$, $r=1,\ldots,n-1$, are algebraic integers.
\end{lemma}

\begin{proof}
We use the notation from Lemma \ref{lem:growcomp} for the coefficients of $f_1,\ldots,f_s$. Let $\alpha \in \K^c$ and $i_1,\ldots,i_n \in \{1,\ldots,s\}$ be such that $\beta := f_{i_1 \ldots i_n}(\alpha)$ is an algebraic integer. For any non-archimedean place $v$ of $\K$, we must have
$$
|\alpha|_v, |f_{i_1\ldots i_r}(\alpha)|_v\le\max_{\substack{i=1,\ldots,s\\j=0,\ldots,d_i-1}} \{1,|a_{i,j}|_v |a_{i,d_i}|_v^{-1}, |a_{i,d_i}|_v^{-1}\}
$$ 
for all $r=1,\ldots,n-1$. Otherwise, by Lemma~\ref{lem:growcomp}, we get that $|\beta|_v > 1$, which contradicts the fact that $\beta$ is an algebraic integer.

Hence, taking $D$ to be a positive integer such that $Da_{i,d_i}^{-1}$ and $Da_{i,j}a_{i,d_i}^{-1}$, 
$i=1,\ldots,s$ and $j=0,\ldots,d_i-1$, are all algebraic integers, we conclude that $D \alpha$ and $D f_{i_1 \ldots i_r}(\alpha)$, $r=1,\ldots,n-1$, are also algebraic integers.
\end{proof}

\subsection{Growth of the number of terms in rational function iterates} The main result of the paper also relies on the following result of Fuchs and Zannier~\cite[Theorem~2]{FZ} which gives a lower bound for the number of terms in a decomposable Laurent polynomial.

\begin{lemma}
\label{lem:FZ}
Let $\ell$ be a positive integer and let $h\in\K[X,X^{-1}]$ be a Laurent polynomial with $\ell$ non-constant terms. Suppose that $h(X)=g(q(X))$, where $g\in\K[X]$, $q\in\K[X,X^{-1}]$ and where $q$ is not of the form $aX^n+bX^{-n}+c$ for $a,b,c\in\K$, $n\in\N$. Then
$$
\deg g\le 2(2\ell-1)(\ell-1).
$$ 
\end{lemma}

\section{Proofs of main results}

\subsection{Proof of Theorem~\ref{thm:KcOrb}}

Since there are finitely many $\gamma \in \cH_A$ in the images of $f_1,\ldots,f_s$, it suffices to fix $\gamma \in \cH_A$, and prove that there are finitely many $\alpha \in \K^c$ such that $\gamma \in \cF_k(\alpha)$ for some $k \geq 1$, where $\cF_k(\alpha)$ is defined by~\eqref{eq:tree alpha}. Assume for contradiction that there are infinitely many such $\alpha \in \K^c$. Let $M$ be a positive integer satisfying
\begin{equation}
\label{eq:m}
M > \frac{2 \log ( 2 L_\K(DL))}{\log 2} + 3.
\end{equation}

If $k\le M$, then obviously there are finitely many $\alpha\in \K^c$ such that $f_{i_1\ldots i_k}(\alpha)=\gamma$ for any $i_1,\ldots,i_k\in\{1,\ldots,s\}$. So there must exist infinitely many $\alpha\in\K^c$ such that $\gamma \in \cF_{\ell}(f_{i_1\ldots i_M}(\alpha))$ for some $\ell\ge 1$ and some fixed $i_1,\ldots,i_M\in\{1,\ldots,s\}$.

Thus, by Lemma \ref{lem:algint}, there exists a positive integer $D$ depending only on $f_1,\ldots,f_s$ such that $D \cdot f_{i_1 \ldots i_M}(\alpha)$ is an algebraic integer for infinitely many $\alpha \in \K^c$. Moreover, where $L$ is defined as in Corollary \ref{cor:growhouse}, the house of $D \cdot f_{i_1 \ldots i_M}(\alpha)$ for such $\alpha$ is at most $DL$. Therefore, 
$$\# \{ \alpha \in \K^c \mid D \cdot f_{i_1 \ldots i_M}(\alpha) \in \cH_{DL} \} = \infty,$$ 
and so, by Lemma~\ref{lem:chen}, there exists a rational function 
$S \in \K^c(X)$, integers $n_i$, roots of unity $\beta_i \in \U$, and $e_i \in E$ (where $E$ is as in Lemma~\ref{lem:loxton}) such that
$$
D \cdot f_{i_1 \ldots i_M}(S(X)) = \sum_{i=1}^b \beta_i e_i X^{n_i},
$$
where $b \le L_\K(DL)$ and $L_\K$ is defined by~\eqref{eq:LK}.

We now want to apply Lemma~\ref{lem:FZ}, either with $g=D \cdot f_{i_1 \ldots i_M}$ and $q = S$, $g = D \cdot f_{i_2 \ldots i_M}$ and $q = f_{i_1} \circ S$, $g = D \cdot f_{i_3 \ldots i_M}$ and $q = f_{i_1 i_2} \circ S$ or $g = D \cdot f_{i_4 \ldots i_M}$ and $q = f_{i_1 i_2 i_3 } \circ S$ to conclude that we must have
$$
2^{M-3} \le \deg g \le 2(2b-1)(b-1) \le 4L_\K(DL)^2,
$$
which contradicts our choice of $M$.
It is enough to show that one of $S$, $f_{i_1} \circ S$, $f_{i_1i_2} \circ S$ and $f_{i_1i_2i_3} \circ S$ is not of the forbidden form. Assume all have this form, that is 
\begin{align} \label{eq:trinom}
S(X) & = a_1X^{m_1}+b_1X^{-m_1}+c_1, \notag \\
f_{i_1}(S(X)) & = a_2X^{m_2}+b_2X^{-m_2}+c_2\\
f_{i_2}(f_{i_1}(S(X)))&=a_3X^{m_3}+b_3X^{-m_3}+c_3 \notag\\
f_{i_3}(f_{i_1i_2}(S(X)))&=a_4X^{m_4}+b_4X^{-m_4}+c_4 \notag.
\end{align}
Note that we must have $m_2 = d_1 m_1$, $m_3=d_2 m_2$ and $m_4=d_3m_3$, where $d_k=\deg(f_{i_k})$, $k=1,2,3$. 

{\it Case I: $a_1 b_1 \neq 0$.} Let $Y = u X^{m_1}$, where $u^2 = a_1 b_1^{-1}$. Then $S(X) = L_1(Y+Y^{-1})$, where $L_1(X) = a_1 u^{-1} X + c_1$. Hence, by \eqref{eq:trinom}, 
$$
f_{i_1}(L_1(Y+Y^{-1})) = a_2 u^{-d_1} Y^{d_1} + b_2 u^{d_1} Y^{-d_1} + c_2.
$$
Since the left-hand side of this equation is invariant under $Y \mapsto Y^{-1}$, we obtain $a_2 u^{-d_1} = b_2 u^{d_1}$, and so 
\begin{equation}
\label{eq:fi1}
f_{i_1}(L_1(Z)) =a_2u^{-d_1}\(Y^{d_1}+Y^{-d_1}\)+c_2= L_2(T_{d_1}(Z)),
\end{equation}
where $Z=Y+Y^{-1}$ and $L_2(X) = a_2 u^{-d_1} X + c_2$.  

From~\eqref{eq:trinom} and~\eqref{eq:fi1}, we have
$$
f_{i_2}(L_2(T_{d_1}(Z)))=a_3u^{-d_1d_2}Y^{d_1d_2}+b_3u^{d_1d_2}Y^{-d_1d_2}+c_3.
$$
As above, since the left-hand side  is invariant under $Y \mapsto Y^{-1}$, we obtain $a_3u^{-d_1d_2}=b_3u^{d_1d_2}$, and thus
$$
f_{i_2}(L_2(T_{d_1}(Z)))=a_3u^{-d_1d_2}\(Y^{d_1d_2}+Y^{-d_1d_2}\)+c_3=L_3(T_{d_1d_2}(Z)),
$$
where $L_3(X)=a_3u^{-d_1d_2}X+c_3$. 
Now, since $T_{d_1d_2}(Z)=T_{d_2}(T_{d_1}(Z))$, we have
\begin{equation}
\label{eq:fi2}
f_{i_2}(L_2(T))=L_3(T_{d_2}(T)), \quad T=T_{d_1}(Z).
\end{equation}

Using the last relation in~\eqref{eq:trinom}, similar computation as above shows that there exists another linear polynomial $L_4$ such that
\begin{equation}
\label{eq:fi3}
f_{i_3}(L_3(U))=L_4(T_{d_3}(U)), \quad U=T_{d_2}(T).
\end{equation}

If $i_1=i_2=i_3$, then from~\eqref{eq:fi1},~\eqref{eq:fi2} and ~\eqref{eq:fi3} we conclude that
$$
f_{i_1}^{(3)}=L_4\circ T_{d_{i_1}^3}\circ L_1^{-1}.
$$
Applying now Lemma~\ref{lem:mon Cheb}, we obtain a contradiction with the first condition of Definition~\ref{def:Srf Sys}. If at least two of $i_1,i_2$ and $i_3$ are distinct, we obtain a contradiction with the second condition of Definition~\ref{def:Srf Sys}.  

Also, we note that in the above we could have considered only $f_{i_1}$ and $f_{i_2}$ if $d_1,d_2>2$ since Lemma~\ref{lem:mon Cheb} (ii) applies in this case.

{\it Case II: $a_1b_1 =0$.} We may assume that $b_1 = 0$, and so $S(X) = L_4(Z)$, where $Y=X^{m_1}$ 
and $L_5(X) = a_1X+c_1$. Since Lemma~\ref{lem:mon Cheb} (i) applies for degree two as well, it is enough to consider only the first three relations in~\eqref{eq:trinom}.

Then, from \eqref{eq:trinom}, 
$$
f_{i_1}(L_5(Y)) = a_2Y^{d_1}+b_2Y^{-d_1}+c_2.
$$
Since the left-hand side is a polynomial in $Y$, we have $b_2=0$, and thus $f_{i_1}(L_5(Y)) = L_6(Y^{e_1})$, where $L_6(X) = a_2 X + c_2$. Similarly, we obtain $f_{i_2}(L_6(U))=L_7(U^{d_2})$, where $U=Y^{d_1}$ and $L_7(X)=a_3X+c_3$. As in the previous case, we  obtain again a  contradiction with  the fact that the set $\cF$ is non-special.

\subsection{Proof of Corollary~\ref{cor:r o u}} As remarked in Section~\ref{sec:def}, $\cH_1=\U$. To  apply  Theorem~\ref{thm:KcOrb} and conclude the proof, we need to see that $f_i(\K^c)\cap \U$ is finite for all $i=1,\ldots,s$. Since, by our hypothesis, the polynomials $f_i(X)-Y^{m_i}$ as polynomials in $X$ do not have a root in $\K^c(Y)$ for all $m_i\le d_i$, this follows directly from the proof of the cyclotomic version of the Hilbert's Irreducibility Theorem due to Dvornicich and Zannier~\cite[Corollary 1]{DZ}.

\subsection{Proof of Theorem~\ref{thm:KcPreper}}

Suppose $\Pi(\cF)$ is infinite, and let $M$ be a positive integer defined as in \eqref{eq:m} depending only on the polynomials $f_1,\ldots,f_s$ and $\K$. Since the number of solutions to $f_{i_1 \ldots i_n}(\alpha) = f_{i_1 \ldots i_k}(\alpha)$ with $M \ge n > k$ and $i_1, \ldots, i_n \in \{1, \ldots, s\}$ is finite, we must have infinitely many $\alpha \in \K^c$ satisfying the above with $n>M$.  Hence, there exist some fixed indices $i_1, \ldots, i_M \in \{1, \ldots, s \}$ such that there are infinitely many $\alpha \in \K^c$ with $f_{i_1 \ldots i_k}(\alpha) \in \cF_\ell(f_{i_1 \ldots i_M}(\alpha))$ 
for some $k < \ell+M$, $\ell \geq 1$. Consider such an $\alpha$, say $$f_{i_1 \ldots i_k}(\alpha) = f_{i_1 \ldots i_M \ldots i_{M+\ell}}(\alpha)=f_{i_{M+1} \ldots i_{M+\ell}}(f_{i_1 \ldots i_{M}}(\alpha))$$ 
with $i_{M+1},\ldots, i_{M+\ell} \in \{1,\ldots,s\}$.

If $k\le M$, then composing the above identity with $f_{i_{k+1} \ldots i_{M}}$ we obtain
$$f_{i_1 \ldots i_M}(\alpha) =f_{i_{M+1} \ldots i_{M+\ell}i_{k+1} \ldots i_{M}}(f_{i_1 \ldots i_{M}}(\alpha)).$$ 
Thus, $f_{i_1 \ldots i_{M}}(\alpha)\in\Pi(\cF)\cap\K^c$.

Let $L$ be defined as in Corollary \ref{cor:growhouse}, say with $A=1$. Then we must have $\house{f_{i_1 \ldots i_M}(\alpha)} \leq L$. Otherwise, as in the proof of Corollary \ref{cor:growhouse}, there exists a conjugate $\sigma(f_{i_1 \ldots i_M}(\alpha))$ of $f_{i_1 \ldots i_M}(\alpha)$ with $|\sigma(f_{i_1 \ldots i_M}(\alpha))| > L$, and so we deduce from Lemma \ref{lem:growabsval} that 
\begin{align*}
|\sigma(f_{i_1 \ldots i_M}(\alpha))| & = |\sigma\(f_{i_{M+1} \ldots i_{M+\ell}i_{k+1} \ldots i_{M}}(f_{i_1 \ldots i_{M}}(\alpha))\)|  > |\sigma(f_{i_1 \ldots i_M}(\alpha))|,
\end{align*}
and thus a contradiction.

Moreover, similar to Lemma \ref{lem:algint}, we have for any finite place $v$ of $\K$,
$$
|f_{i_1 \ldots i_M}(\alpha)|_v \le\max_{\substack{i=1,\ldots,s\\j=0,\ldots,d_i-1}} \{1,|a_{i,j}|_v |a_{i,d_i}|_v^{-1}, |a_{i,d_i}|_v^{-1}\},
$$
as otherwise $\{ | f_{i_1 \ldots i_r}(\alpha) |_v \}_{r=M}^\infty$ is strictly increasing by Lemma~\ref{lem:growcomp}. 
Hence, if $D$ is a positive integer such that $Da_{i,d_i}^{-1}$ and $Da_{i,j}a_{i,d_i}^{-1}$, $i=1,\ldots,s$ and $j=0,\ldots,d_i-1$, are all algebraic integers, then we obtain that $D\cdot f_{i_1 \ldots i_M}(\alpha)$ is an algebraic integer.

If $M<k<\ell+M$, then since $f_{i_1 \ldots i_{k}}(\alpha)\in\Pi(\cF)\cap\K^c$, as above, we obtain that $\house{f_{i_1 \ldots i_k}(\alpha)} \leq L$ with $L$ defined as in Corollary \ref{cor:growhouse}. Applying now Corollary \ref{cor:growhouse}, and the fact that $M<k$, we obtain that $\house{f_{i_1 \ldots i_M}(\alpha)} \leq L$. 
Moreover, as above,  for any finite place $v$ of $\K$,
$$
|f_{i_1 \ldots i_k}(\alpha)|_v \le\max_{\substack{i=1,\ldots,s\\j=0,\ldots,d_i-1}} \{1,|a_{i,j}|_v |a_{i,d_i}|_v^{-1}, |a_{i,d_i}|_v^{-1}\}.
$$
However, since  $M<k$, we also have that 
$$
|f_{i_1 \ldots i_M}(\alpha)|_v \le\max_{\substack{i=1,\ldots,s\\j=0,\ldots,d_i-1}} \{1,|a_{i,j}|_v |a_{i,d_i}|_v^{-1}, |a_{i,d_i}|_v^{-1}\}
$$
since otherwise, by Lemma \ref{lem:growcomp}, we have
$$
|f_{i_1 \ldots i_k}(\alpha)|_v>|f_{i_1 \ldots i_M}(\alpha)|_v>\max_{\substack{i=1,\ldots,s\\j=0,\ldots,d_i-1}} \{1,|a_{i,j}|_v |a_{i,d_i}|_v^{-1}, |a_{i,d_i}|_v^{-1}\},
$$
and thus a contradiction.
We finally obtain that $D \cdot f_{i_1 \ldots i_M}(\alpha)$ is an algebraic integer for a positive integer $D$ as above.

We conclude that $$\# \{ \alpha \in \K^c \mid D \cdot f_{i_1 \ldots i_M}(\alpha) \in \cH_{DL} \} = \infty,$$ and from this point the proof to obtaining a contradiction with Definition~\ref{def:Srf Sys} is the same as in the proof of Theorem~\ref{thm:KcOrb}.

\subsection{Proof of Theorem~\ref{thm:KcPreper2}}

To prove Theorem~\ref{thm:KcPreper2}, we begin by proving that $\Sigma_A(\cF)$ is a set of bounded house.

Let $\alpha \in \Sigma_A(\cF)$, say 
$$
f_{i_1 \ldots i_n}(\alpha) = \sum_{k=0}^{n-1} \sum_{j_1,\ldots,j_k=1}^s \gamma_{j_1 \ldots j_k} f_{j_1 \ldots j_k}(\alpha),
$$ 
with each $\gamma_{j_1 \ldots j_k} \in \cH_{A^{d^{n-1}}}$. Write $f_i(X)=a_{i,d}X^d+\ldots+a_{i,0}$ for $i=1,\ldots,s$, and choose $m \geq 1$ such that $|\sigma(a_{i,d})| > \frac{1}{m}$ for all embeddings $\sigma$ of $\K$ in $\C$, and $i=1,\ldots,s$. We define
\begin{equation*}
\begin{split}
K &= \max_{\sigma} \left\{ 2s m^2 A +\max_{i=1, \ldots, s} \left(|\sigma(a_{i,d})|+ \l( 1 + \l( |\sigma(a_{i,d})|-\frac{1}{m} \r)^{-1} \r)\right.\right.\\
&\qquad\qquad\qquad\qquad\qquad\qquad\qquad\qquad\qquad\qquad\quad\cdot \left.\left. \sum_{j=0}^{d-1}|\sigma(a_{i,j})|  \right)\right\},
\end{split}
\end{equation*}
where the maximum runs over all embeddings of $\K$ in $\C$. Fix such an embedding $\sigma$, and suppose that $|\sigma(\alpha)| > K$. Then
$$
|\sigma(a_{i_1,d})||\sigma(\alpha)| - \sum_{j=0}^{d-1} |\sigma(a_{i_1,j})| > \frac{|\sigma(\alpha)|}{m},
$$
and so the calculation in \eqref{eq:growabs} gives
$$
|\sigma(f_{i_1})(\sigma(\alpha))| > |\sigma(\alpha)|^{d-1} \cdot \frac{|\sigma(\alpha)|}{m} = \frac{|\sigma(\alpha)|^d}{m} > K.
$$
Iterating this, we obtain
\begin{equation} \label{eq:lowbound}
|\sigma(f_{i_1 \ldots i_n})(\sigma(\alpha))| > \frac{|\sigma(\alpha)|^{d^n}}{m^{1+\ldots+d^{n-1}}} > \frac{|\sigma(\alpha)|^{d^n}}{m^{2d^{n-1}}}.
\end{equation}
On the other hand, since $|\sigma(\alpha)| > 1$,  for any $k < n$ and $j_1,\ldots, j_k \in \{1, \ldots s \}$ we have
\begin{align*}
|\sigma(f_{j_1})(\sigma(\alpha))| & \leq |\sigma(a_{j_1,d})||\sigma(\alpha)|^d+ \ldots + |\sigma(a_{j_1,0})| \\
& \leq \l( \sum_{\ell=0}^d |\sigma(a_{j_1,\ell})| \r) |\sigma(\alpha)|^d.
\end{align*}
Noting again that $|\sigma(f_{j_1})(\sigma(\alpha))| > K>1$, induction yields
\begin{align}
|\sigma(f_{j_1 \ldots j_k})(\sigma(\alpha))| & \leq \prod_{i=1}^k \l( \sum_{\ell=0}^d |\sigma(a_{j_i,\ell})| \r)^{d^{k-i}} |\sigma(\alpha)|^{d^k} \notag \\
& < |\sigma(\alpha)|^{1+\ldots+d^{k}} < |\sigma(\alpha)|^{2d^{k}}. \label{eq:upbound}
\end{align}
We thus obtain
\begin{align*}
|\sigma(f_{i_1 \ldots i_n}) (\sigma(\alpha))| & = | \sigma( f_{i_1 \ldots i_n}(\alpha)) | = \left | \sum_{k=0}^{n-1} \sum_{j_1,\ldots, j_k = 1}^s \sigma( \gamma_{j_1 \ldots j_k} ) \sigma( f_{j_1 \ldots j_k} (\alpha) ) \right| \\
& \leq A^{d^{n-1}} \sum_{k=0}^{n-1} \sum_{j_1,\ldots, j_k = 1}^s | \sigma(f_{j_1 \ldots j_k})(\sigma(\alpha)) | \\
& \leq A^{d^{n-1}} \sum_{k=0}^{n-1} \sum_{j_1,\ldots, j_k = 1}^s | \sigma(\alpha) |^{2d^k} \leq 2A^{d^{n-1}}s^{n-1} |\sigma(\alpha)|^{2d^{n-1}}.
\end{align*}
This contradicts \eqref{eq:lowbound}, as noting that $d \geq 3$, $n,s \geq 1$ and $|\sigma(\alpha)| > K > 2sm^2A$, we have
\begin{align*}
\frac{|\sigma(\alpha)|^{d^n}}{m^{2d^{n-1}}} & = \frac{|\sigma(\alpha)|^{d^n-2d^{n-1}}}{m^{2d^{n-1}}} |\sigma(\alpha)|^{2d^{n-1}} > \left( \frac{|\sigma(\alpha)|}{m^2} \right)^{d^{n-1}} |\sigma(\alpha)|^{2d^{n-1}} \\
& > (2sA)^{d^{n-1}} |\sigma(\alpha)|^{2d^{n-1}} > 2A^{d^{n-1}}s^{n-1} |\sigma(\alpha)|^{2d^{n-1}}.
\end{align*}
We conclude that $\house{\alpha} \leq K$.

Furthermore, suppose that
$$
|\alpha|_v > \max_{\substack{i=1,\ldots,s\\j=0,\ldots,d-1}} \{|a_{i,j}|_v |a_{i,d}|_v^{-1}, |a_{i,d}|_v^{-1}, |a_{i,d}|_v\},
$$
for some finite place $v$ of $\K$. Then \eqref{eq:grownonarch} gives
$$
|f_{i_1}(\alpha)|_v = |a_{i_1,d}|_v |\alpha|_v^d > |\alpha|_v.
$$
Hence,
\begin{align*}
|f_{i_1 \ldots i_n}(\alpha)|_v & = |a_{i_n,d}|_v |a_{i_{n-1},d}|_v^d \ldots |a_{i_1,d}|_v^{d^{n-1}} |\alpha|_v^{d^n} \\
& = (|a_{i_n,d}|_v |\alpha|_v) (|a_{i_{n-1},d}|_v |\alpha|_v)^d \ldots (|a_{i_1,d}|_v |\alpha|_v)^{d^{n-1}} |\alpha|_v^{d^n-\frac{d^n-1}{d-1}}.
\end{align*}
Since $d \geq 3$, for any $k < n$ we have
$$
d^n - \frac{d^n-1}{d-1} = \frac{d^{n+1}-2d^n-1}{d-1} \geq \frac{d^{k+1}-1}{d-1},
$$
and so for any $j_1, \ldots j_k \in \{ 1, \ldots, s\}$,
\begin{align*}
|f_{i_1 \ldots i_n}(\alpha)|_v & > |\alpha|_v^{\frac{d^{k+1}-1}{d-1}} = \l( |\alpha|_v |\alpha|_v^d \ldots |\alpha|_v^{d^{k-1}} \r) |\alpha|_v^{d^k} \\
& > |a_{j_k,d}|_v |a_{j_{k-1},d}|_v^d \ldots |a_{j_1,d}|_v^{d^{k-1}} |\alpha|_v^{d^k} = |f_{j_1 \ldots j_k}(\alpha)|_v.
\end{align*}
This contradicts the fact that
\begin{align*}
|f_{i_1 \ldots i_n}(\alpha)|_v & \leq \max_{0 \leq k \leq n-1} \left \{ \max_{1 \leq j_1, \ldots, j_k \leq s} |\gamma_{j_1 \ldots j_k}|_v |f_{j_1 \ldots j_k}(\alpha)|_v \right \} \\
& \leq \max_{0 \leq k \leq n-1} \left \{ \max_{1 \leq j_1, \ldots, j_k \leq s} |f_{j_1 \ldots j_k}(\alpha)|_v \right \},
\end{align*}
recalling that  $\gamma_{j_1 \ldots j_k}$ are algebraic integers.

Therefore, similarly to Lemma \ref{lem:algint}, if $D$ is a positive integer such that $Da_{i,d}^{-1}$, $Da_{i,j}$ and $Da_{i,j} a_{i,d}^{-1}$, $i=1,\ldots,s$, $j=0,\ldots,d$, are all algebraic integers, then $D\alpha$ is an algebraic integer.

\section{Final comments and questions}

Certainly we would like to find explicit conditions on the polynomials $f_1,\ldots,f_s\in\K[X]$ such that the set $\Sigma_A(\cF)\cap \K^c$ is finite. Moreover, Theorem~\ref{thm:KcPreper2} applies only to polynomials having the same degree $d\ge 3$, thus obtaining an analogue for polynomials of different degrees, including degree two, would be of great interest.

A particular case would be to describing  polynomials $f_1,f_2\in\K[X]$ such that
$$
\# \{ \alpha \in \K^c \mid f_1^{(n)}(\alpha)/f_2^{(k)}(\alpha) \in \U \textrm{ for some $n,k\ge 1$} \} < \infty.
$$
(We note that for preperiodicity we only need $n\ne k$, however it would be interesting to consider the case $n=k$ as well.)

We note that the approach of the results in this paper seem not to apply in this case.

We are also interested in extending such results to division groups. Let $\Gamma$ be a finitely generated subgroup of $\K^*$ and let $\K_\Gamma=\K(\ov\Gamma)$, where $\ov\Gamma$ is the division group of $\Gamma$, that is,
$$
\ov\Gamma=\{\alpha\in\ov\Q \mid \alpha^n \in \Gamma \textrm{ for some $n\ge 1$ }\}.
$$
We note that when $\Gamma=\{1\}$, then $\ov\Gamma=\U$ and $\K_\Gamma=\K^c$.

\begin{problem}
Let $\cF=\{f_1,\ldots,f_s\}\in\K[X]$. Under what conditions the set $\ov\Pi(\cF)\cap \K_\Gamma$ is finite?
\end{problem}

The first step of course is studying this for preperiodic points of one polynomial, before attempting to tackle this in the semigroup case.

\section*{Acknowledgement}
The authors are very grateful to Igor Shparlinski for many valuable discussions on the topic of the paper and beyond, and for suggesting considering the set $\Sigma_A(\cF)$ in this general form. The authors are also grateful to the anonymous referee for bringing to their attention more references and connections to other works, as well as useful comments that improved the presentation of the paper.

This work  was supported  in part by
the  Australian Research Council  Grant  DP180100201.

\end{document}